\documentclass[11pt]{article}
\usepackage{mylatexsetting}

\def\eL{K}

\title{On a connection between total positivity and Bernoulli stopping problems}
\author{Zakaria Derbazi\footnote{Correspondence address: zakaria@derbazi.net} \\{\it\small Queen Mary University of London}}

\def\myabstract{
	Consider a discrete-time optimal selection problem where one observes a sequence of   independent Bernoulli trials and receives a nonnegative reward upon stopping on a success. The objective is to find a single-choice strategy that maximises the expected payoff. These Bernoulli stopping problems are characterised by two key properties: (i) a recurrence relation connecting the reward sequence to the continuation payoff sequence, and (ii) the total positivity of the Markov chain  embedded in success epochs of the trials. The recurrence is fundamental in proving the optimality of the myopic strategy under unimodal continuation payoff sequence. On the other hand, the total positivity ensures that the expectation of a quasi-unimodal function of the chain remains quasi-unimodal with respect to the initial state. In particular, if  the number of successes is finite almost surely, the quasi-unimodality of the  reward sequence is sufficient for the myopic rule to be optimal. Illustrative examples are given in various last-success settings.
}
\usetikzlibrary{arrows.meta,automata, chains, positioning, arrows}
\tikzset{%
	>={Latex[width=2mm,length=2mm]},
	base/.style = {rectangle, rounded corners, draw=black,
		minimum width=2cm, minimum height=1cm,
		text centered, font=\sffamily},
	activityStarts/.style = {base},
	startstop/.style = {base, fill=green!5},
	startstop2/.style = {base, fill=purple!5},	
	activityRuns/.style = {base, fill=green!30},
	activityRuns2/.style = {base, fill=purple!40},	
	finalOutput/.style = {base, fill=black!30},	
	process/.style = {base, minimum width=2.5cm, fill=orange!15,
		font=\ttfamily},
	process2/.style = {base, minimum width=2.5cm, fill=orange!50,
		font=\ttfamily},		
}
\begin{document}
	\maketitle
	\begin{abstract}
		\myabstract
	\end{abstract}

	\section{Introduction}\label{sect1}
	Let  $f \coloneqq (f_k)_{k \ge 0}$ be a bounded sequence of nonnegative rewards associated with the following stopping game:    a player observes a sequence of independent Bernoulli trials arriving at discrete epochs. Upon encountering a success, at trial $k \in \mn$,  the player faces two choices: 
	1) stop and collect the reward $f_k$, in which case, the game terminates, or 2) proceed to the next trial. We set $f_0 = 0$ and $f_\infty= 0$ to represent the reward for not taking any observation and the reward for never stopping,  respectively. The latter ensures that the game will stop with probability one.

	More formally, consider a sequence of  Bernoulli random variables  $X_1, X_2, \ldots$ defined on a probability space $(\Omega, {\cal B}, \prob)$. Each trial has success probability  $p_k \coloneqq \prob(X_k=1)$ and the collection of success probabilities $p \coloneqq(p_k)_{ k \ge 1}$ is referred to as the \textit{success profile}.	Denote the random reward sequence of the stopping problem by  $Y\coloneqq(Y_k)_{k \ge 0}$, where  $Y_0 = 0 \;\text{almost surely} \;(\rm a.s.)$ and 
	\begin{equation}\label{def.payoof}
			Y_k \coloneqq  \1\{X_k = 1\} f_k
	\end{equation}
	is the reward for stopping at trial $k \in \mn$. Here,	$\1\{\cdot\}$ denotes the indicator function.	Each $Y_k$ is $\mathcal{F}_k$-measurable, where  $\mathcal{F}_k$ denotes the $\sigma$-algebra  generated by the first $k$ trials satisfying  $\mathcal{F}_{k-1} \subset \mathcal{F}_{k}$ for every $k \in \mn$, with ${\cal F}_0=\{ \varnothing, \Omega\}$ and  $\mathcal{F}_\infty$ being the $\sigma$-algebra generated by $\cup_k \mathcal{F}_k\subset \cal B$.  Let $\mathcal{T}$ denote the class of all nonanticipating stopping rules with respect to  $(\mathcal{F}_k)_{k \ge 0}$. To solve the stopping problem, we seek a rule $\tau$ in $\mathcal{T}$ that maximises the value of the game; that is,  $\mE{Y_\tau} =\,\sup_{k \in \calT} \mE{Y_k}$. Although each $Y_k$ has an existing and finite expectation, the supremum may not be well-defined. To overcome this difficulty, we employ   the \textit{Snell envelope} of $Y$, denoted  $Y^\circ\coloneqq(Y^\circ_k)_{ k \ge 0}$, where $	Y^\circ_k \coloneqq  \esssup_{j > k} Y_j$ for $k \in \mn$. We proceed by relying on the familiar principle of optimality to check if the problem is  \textit{monotone} \cite{CRS}. This criterion holds true if the sequence of events $A \coloneqq (A_k)_{k \ge 0}$, where $A_k \coloneqq \{ Y_j \ge \mE{Y_{j+1} | \mf_k}\}$,  satisfies the condition $$A_{0} \subset A_{1}\subset A_{2} \subset \cdots ~\;\text{(a.s.)}$$ In this case, the sequence $A$ is said to be  \textit{absorbing} and	the \textit{myopic} strategy---also referred to as the \textit{one-step lookahead} strategy---is optimal \cite{FergusonBook}. Using the convention $\min \{\varnothing\} = \infty$, the myopic rule  is  given by
	\begin{equation*}
		\tau \coloneqq \min\{ j \ge 1:  Y_j \ge \mee Y^\circ_j \}.
	\end{equation*}
The scenario $\tau=\infty$ corresponds to the strategy of never stopping, which yields the reward $f_\infty$. 
	Define $g \coloneqq \mE{Y^\circ} =(g_k)_{k \ge 0}$ and assume that this sequence is either known directly or obtained via the relation $g= \eL f$. Here, 	$\eL$ represents the one-step transition operator that, at stage $k$, computes the expected payoff $g_{k+1}$ obtained by skipping the first $k$ trials and then stopping at the first success among $X_{k+1},X_{k+2}, \ldots$ (if any).   Under this framework,  the event $A_k$ can be written as $\{ X_k=1 \text{ and } f_k \ge g_{k} \}$ and the myopic rule becomes
	\begin{equation}\label{ch1.def.ola.rule}
		\tau = \min\{ j \ge 1: X_j=1 \text { and }  f_j \ge g_{j} \}.
	\end{equation}
	This particular form of the myopic rule characterises	 the class of \textit{Bernoulli stopping problems}, which generalise the relatively-best choice problems introduced by \cite{FergusonHardwick}.

	Bernoulli stopping problems are rooted in the foundational work on discrete-time stopping problems developed by \cite{Dynkin}  and \cite{FergusonBook}, respectively. These problems generalise the classical secretary problem and several of its extensions, including the duration problem \cite{FergusonHardwick} and the last-success problem  \cite{GD2}. A special case relevant  to our work is the model studied by  \cite{Ribas2019}, who examined the last-success problem under arbitrary rewards and a finite number of trials.  Bernoulli stopping problems are specified by a success profile $p$ and a sequence of fixed  \textit{(stopping) rewards} $f$, or equivalently by  a  sequence of {rewards} $f$  together with either a sequence of  \textit{continuation} payoffs $g$ or  a transition operator $\eL$. 
	
	Our goal in this work is to revisit the fundamental question of optimality of the myopic strategy, particularly  in the infinite-horizon setting with quasi-unimodal payoffs. A quasi-unimodal sequence is one that may or may not attain its supremum and becomes unimodal once it is  truncated. This investigation is  motivated in part by analysing the intrinsic Markovian structure of Bernoulli stopping problems and extending existing results to the case of infinite number of trials. In addition, we aim to study the properties of the transition operator that links stopping rewards to continuation payoffs. This connection is explored through the theory of total positivity, which has found successful applications in many areas of mathematics \cite{KarlinTPBook}.
	
	Recall that a real matrix, finite or infinite, is   \textit{sign regular} of order $r$,  or $\SR{r}$ for short, if the signs of all minors of order $k=1,\ldots, r$ depend only on $k$. In the special case where these minors are nonnegative, the matrix is called   \textit{totally positive},  abbreviated as $\TP{r}$. Similarly, a sequence  $a \coloneqq (a_k)_{k \in I}$, where  $I  \subseteq \mn$,  is  $\SR{r}$ (resp. $\TP{r}$)  if its \textit{kernel representation}---the matrix $A\coloneqq(a_{i-j})_{i,j \in I}$, where $a_k = 0$ when $k < 0$---is  $\SR{r}$ (resp. $\TP{r}$)  \cite{KarlinTPBook}. 
	When the property $\SR{r}$ (resp. $\TP{r}$) holds for all orders, the subscript $r$ is omitted, and the matrix or sequence is simply called sign-regular (resp. totally positive).  An important feature of the class of \( \TP{} \) matrices is their closure under matrix multiplication.

	Transformations involving sign-regular matrices are characterised by the \textit{variation-diminishing} (VD) property: if a transformation $K$ defined on a sequence $u$, possesses this property, then the number of sign changes in  $Ku$  does not exceed the 
	number of sign changes in $u$. 
	
	To verify a unimodality-preserving property in  Markov processes with a stochastic monotone transition operator,   \cite{KeilsonKester} combined  the theory of total positivity with the properties of stochastically monotone matrices---first introduced by \cite{Kalmykov} and later generalised by \cite{Daley}. However, in the countably infinite case, such an operator preserves only quasi-unimodality, since the transformed sequence need not  attain its supremum.  Even in the countably finite case, the method hinges on  stochastic monotonicity of the transition matrix. Consequently, an alternative approach is needed to determine when a stochastic matrix preserves (quasi-)unimodality. It turns out that $\TP{3}$  stochastic matrices preserve quasi-unimodality unconditionally.

	The structure of the paper is as follows: Section 2 introduces key properties of unimodal sequences and shows that transformations induced by  
	$\TP{3}$ stochastic matrices preserve unimodality for finite sequences and quasi-unimodality for infinite sequences. This result finds natural application in the classical embedding of success epochs of independent Bernoulli trials in a Markov chain. Specifically, it is shown that the expectation of unimodal functions of the chain is quasi-unimodal in the initial state. 	Since the continuation payoff sequence is a mapping of the stopping reward sequence via the transition matrix of this chain, 
	Section 3 builds on the previous findings to derive a recurrence relation linking continuation payoffs and stopping rewards. It then establishes a sufficient condition for the optimality of the myopic strategy, based on the unimodality of the continuation payoffs or the quasi-unimodality of the stopping rewards. Finally, Section 4 applies these results to generalised last-success problems, providing new and simpler proofs of the myopic rule's optimality.

	\section{Transformations Preserving Quasi-Unimodality }\label{sec2}
	
	First, we introduce several key concepts used throughout this paper.
		\begin{define}[Minor sign]
		Given a sign-regular matrix, its  minor sign of order $k \ge 1$  is the number $\varepsilon_k \in \{-1, 1\}$ satisfying $\varepsilon_k d_k \ge 0$, where $d_k$ denotes any determinant of order $k$ of the matrix.
	\end{define}
	
	\begin{define}[Number of sign changes]\label{def.sign.chg}
		The number of sign changes in a finite or infinite sequence $u \coloneqq (u_1, u_2, \ldots,)$ is defined by
		\begin{equation*}
			S(u)  \coloneqq  S(u^\prime_{1},u^\prime_{2}, \ldots) = \sum_{i \ge 1} \abs{\sgn(u^\prime_{i+1})-\sgn(u^\prime_i)},\qquad S(0, 0, \ldots, 0) =0,
		\end{equation*}	
		where $u^\prime_1, u^\prime_{2}, \ldots$  is the sequence $u$  with all zero terms removed.
	\end{define}

	\begin{define}[Unimodal sequence]\label{def.unimodel}
		A   real  sequence $u \coloneqq (u_1, u_2, \ldots)$  supported on the lattice of positive integers is said to be 
		\textit{unimodal} or modal of order one, if there exists an \textit{interval of modes}
		$[k^-, k^+] \subset \mn$ such that
		\begin{equation}\label{def.condition.multimodel.sequence1}
			\begin{dcases}
				u_{j} \le  u_{j+1} &~ \text{all } ~ j  <  k^-,\\
				u_{j} =  u_{j+1} &~ \text{all } ~  k^- \le j  < k^+,\\
				u_{j} \ge  u_{j+1} &~ \text{all } ~  j \ge k^+
			\end{dcases}.
		\end{equation}
	\end{define}	
	A more practical characterisation of unimodality is given next.
	\begin{lemma}[\cite{ZD4}]\label{lem.sign.chg}
		A  sequence $u : I \subseteq \mn \to \mr$ of two or more elements is unimodal iff
		\begin{enumerate}[label=\rm(\roman*)]
			\item For every $\lambda \in \mr$, $u-\lambda$ has at most two sign changes and whenever two sign changes occur,  their pattern is  $-, +, -$.
				\item $u$ attains its supremum.
		\end{enumerate}
	\end{lemma}	
Condition \rm(ii), which always holds for finite sequence, implies $u$ possesses a nonincreasing tail. That is, there exists an integer $N \ge 1$ such that $u_{k+1} \le u_{k}$ for all $k \ge N$. When $I \coloneq \{ 1, \ldots, n \}$, we set $u_{n+1}\coloneqq u_n$.  

Many infinite sequences---such as the infinite nondecreasing sequences $u_k=k$  or $u_k = 1-1/k$ for $k > 0$---do not attain a maximum. We refer to such sequences as \textit{quasi-unimodal}.
\begin{define}[Quasi-unimodal sequence]
A real sequence  on $\mn$ is quasi-unimodal if it satisfies condition \rm(i)  of Lemma \ref{lem.sign.chg}.
\end{define}
\begin{lemma}
Every  unimodal real sequence is quasi-unimodal, and every truncated  quasi-unimodal real sequence is unimodal.
\end{lemma}
\begin{proof}
The first assertion is immediate from the definitions. The second follows from the fact that any finite sequence attains its supremum. 
\end{proof}

Having disposed of the preliminary definitions, we now turn to transformations with sign-regular kernels. 	Consider the  linear transformation based on a kernel  $K: \mn \times \mn $ defined for a bounded sequence  $u:\mn\to \mr$
	\begin{alignat}{2}
		v_n &\coloneqq  \sum_{j \in \mn} K(n,j)u_j,\qquad &n \in \mn \label{def.kernel.sequence-sequence},
	\end{alignat}
	This transformation is well-defined when the series $\sum_{j \in \mn} K(n, j)$ converges  for each $n \in \mn$. Our goal is to identify conditions under which this transformation maps a quasi-unimodal sequence to another quasi-unimodal sequence. To achieve this, we rely on the theory  of total positivity to obtain simple, sufficient conditions ensuring that $K$ preserves quasi-unimodality.

	We begin by recalling a  fundamental theorem (see p. 233 in \cite{KarlinTPBook}), which states that under suitable conditions, a transformation governed by a $\SR{r}$ kernel preserves the sign pattern of functions with at most $r-1$ sign changes.
	
	\begin{thm}[Theorem 3.1 \& Proposition 3.1,  Chapter 3 in \cite{KarlinTPBook}]\label{def.vd}
		Fix $r \ge 1$ and suppose  $K$  is the transformation  defined in (\ref{def.kernel.sequence-sequence}), with  $\sum_{j \in \mn} K(n, j) < \infty$ for each $n \in \mn$. If $K$ is $\SR{r}$, then for any  real-valued sequence $u$ on $\mn$ and its mapping 	$v\coloneqq K u$,  the following  hold:
		\begin{enumerate}[label=\rm(\roman*)]
			\item  $K$ is VD whenever $S({f}) \le r-1$.  In particular, if $r=2$ and $u$ is monotone, then $v$ is monotone.
			\item If $S({v})=S({u}) = k \le r-1$ and $K$ is $\SR{r}$, then ${u} $ and ${v}$ exhibit the same monotonicity patterns whenever  $\varepsilon_k\varepsilon_{k+1}=1$.
		\end{enumerate}
	\end{thm}%
	A transformation using a matrix $K \in \TP{2}$  preserves log-concavity (resp. log-convexity) of sequences. In contrast, preserving concavity (resp. convexity) requires $K \in \TP{3}$  \cite{KarlinProschan}. Although positive log-concave sequences are unimodal  \cite{Stanley}, and their unimodality remains intact under a transformation by a $\TP{2}$ matrix, extending this preservation property to arbitrary unimodal sequences is not straightforward. For example, consider the unimodal (and log-convex) sequence $u=(3, 3.5, 3.15)$ and the  $\TP{2}$ matrix $K$, which is not $\TP{3}$, where
	$$
	K=\left(\begin{array}{ccc}
		10 & 5 & 0\\
		8 & 4 & 3\\
		4 & 3 & 8
	\end{array}\right).
	$$
	A direct computation yields $ Ku = (47.5, 47.45, 47.7)$, which has two distinct modes. A further complication for infinite sequences is the possible failure to attain the supremum.  Our first result provides conditions under which a transformation using on a totally positive kernel preserves unimodality.

\begin{thm}\label{thm.unimodality}
Consider  the transformation  (\ref{def.kernel.sequence-sequence}) with a stochastic matrix  $K$. Let $u : \mn \to \mr$ be a quasi-unimodal sequence. If $K  \in \TP{3}$, then  $v \coloneqq K u$ is quasi-unimodal. If in addition $v$ attains its supremum, then $v$ is unimodal. In particular, if $v$ is a finite sequence then it is unimodal.
\end{thm}

	\begin{proof}
	 By assumption $ \sum_{j \ge 1} K(i,j) =1$ for each $i \in \mn$.   Let  $v \coloneqq K u$ and  define the sets $E_k \coloneqq \{ \lambda \in \mr: S(v - \lambda)=k \}$, where $k \in \{ 0, 1, \ldots\}$.  By Lemma \ref{lem.sign.chg} and the VD property, $E_k = \varnothing$ for all $k > 2$,  since for any $\lambda \in \mr$,  $v-\lambda = \eL  (u-{\lambda})$ and 	 $S(u-\lambda) \le 2$. Now, we consider two cases. If $E_2\neq \varnothing$, then by part \rm(ii) of Theorem \ref{def.vd}, $v$ and $u$ share the same $-+-$ sign pattern,  because $K \in \TP{3}$ implies  $\varepsilon_2\varepsilon_3=1$. Hence, $v$ is quasi-unimodal. If $E_2= \varnothing$,then $v$ is monotone and therefore quasi-unimodal. 	 For the second part, if $v$ attains its supremum then according to Definition \ref{def.unimodel} it is unimodal. The final assertion is obvious.
	\end{proof}

	\begin{rem}
		Let  $T$ be the lower triangular matrix of ones and $I^* \coloneqq \diag(0,1,1\ldots,1,)$.  \cite{KeilsonKester} showed that a transformation defined by a stochastic  matrix $Q$ which is \textit{monotone}---that is, $T^{-1}QT \in \TP{2}$---and satisfies the condition $I^*T^{-1}QT \in \TP{2}$, preserves unimodality in the finite case. For infinite matrices, however, such  transformations only preserve quasi-unimodality.
	\end{rem}

	\subsection{The Markov chain embedded in success epochs}
	Let $X \coloneqq (X_1, X_2,\ldots, )$ be a sequence of independent Bernoulli trials with success profile  $p \coloneqq (p_1, p_2,\ldots)$, where $p_j=1-q_j=\prob(X_j = 1)$ and $p_j > 0$  for each trial $j$. 
	Consider the Markov chain $M$  formed by the random indices of successes (Markov times) in the sequence $X$. In this construction, introduced by Dynkin \cite{Dynkin}, failures are ignored, and the sequence of success epochs is represented on the state space ${S}\coloneqq \mn \cup \{0\} \cup \{\infty\}$. Here, $0$ denotes the default initial state, while $\infty$ is the terminal state, which is eventually reached if the number of successes is finite---this occurs with certainty when $\sum_k p_k<\infty$ (By Borel-Cantelli lemma).  For example, given the sequence of trials  $0,1,1,0,0,1,0^\infty$, the corresponding path of $M$ is   $0,2,3,6,\infty$. The dynamics of the chain are  illustrated in Figure \ref{fig.M}.\\
	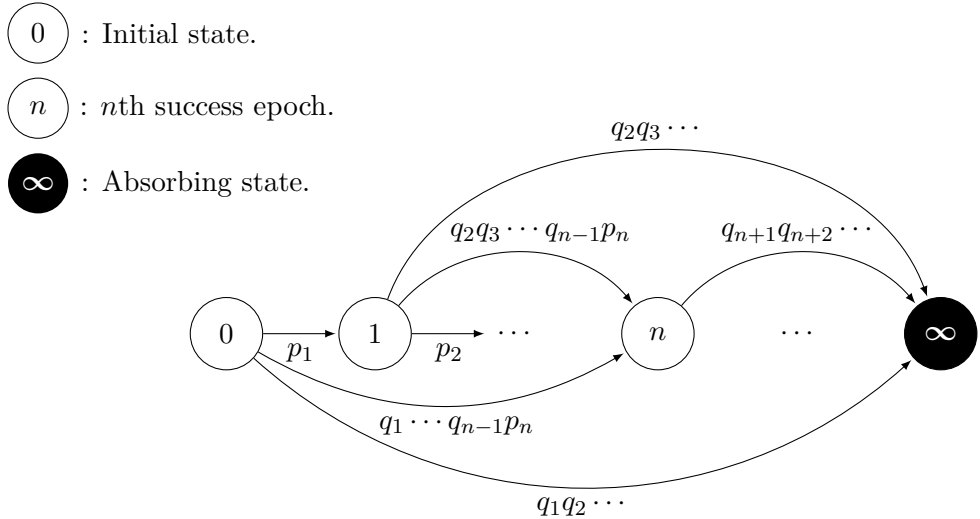
\begin{figure}[htb]
		\centering
		\begin{tikzpicture}[start chain=going left,node distance=1cm]
			\draw (-12,4) circle [radius=0.4] node (initState) {$0$};				
			\draw (-12,3) circle [radius=0.4] node (nState) {$n$};							
			\fill (-12,2) circle [radius=0.4, fill=yellow!80!black] node (absState) {\color{white}$\boldsymbol\infty$};			
			\node[right of=initState, xshift=0.75cm](initialText) {: Initial state.};
			\node[below of=initialText, xshift=0.5cm] {: $n$th success epoch.};
			\node[right of=absState, xshift=1.1cm] {: Absorbing state.};			
			\node[state, on chain, fill opacity=1, fill=black]                 (Infinity) {\color{white}$\boldsymbol\infty$};
			\node[on chain]                   (g2) {$\cdots$};
			\node[state, on chain]                 (N) {$n$};
			\node[on chain]                   (g) {$\cdots$};
			\node[state, on chain]                 (1) {$1$};
			\node[state, on chain]                 (0) {$0$};
			\draw[
			>=latex,
			auto=right,                      
			loop above/.style={out=75,in=105,loop},
			every loop,
			]
			
			(0) edge[bend right = 30,below]	node{\phantom{aa}$q_{1}\cdots q_{n-1}p_n$}	(N)	
			(0) edge[bend right = 42, below]	node{$q_{1}q_2\cdots $}	(Infinity)						
			
			(1) edge[bend left = 50,above]	node{\phantom{ddd}$q_2q_3\cdots q_{n-1}p_n$}	(N)	
			(1) edge[bend left = 70,above]	node{$q_{2}q_{3}\cdots$}	(Infinity)	
			
			(N) edge[bend left = 50,above]	node{$q_{n+1}q_{n+2}\cdots$}	(Infinity)

			(0) edge             node {$p_1$} (1)
			(1) edge             node {$p_2$}   (g);
		\end{tikzpicture}
		\caption{The Markov Chain embedded in success epochs}
		\label{fig.M}
	\end{figure}

	The one-step transition matrix  is given by $P\coloneqq \left(P(i,j)\right)_{0 \le i, j \le \infty}$, where
	\begin{equation}\label{def.embedded.chain.poissonbinom}
		P(i, j) \coloneqq
		\begin{dcases}
			0& ~~~~\text{  } i  \le j,\\
			p_j\prod_{k=i+1}^{j-1} q_k & ~~~~\text{  } i  <j < \infty,\\
			\prod_{k\ge i+1} q_k & ~~~~\text{  } j = \infty.
		\end{dcases} 
	\end{equation}
	By construction, \( P \) is a stochastic matrix. This follows directly  from (\ref{def.embedded.chain.poissonbinom}), since for any state $i \in S$,
	\begin{alignat*}{2}
		\sum_{j  \in { S} \setminus \{ \infty\} } P(i, j)  &=
		\sum_{j > i} \left\{ 
		\prod_{k=i+1}^{j-1} q_k - \prod_{k=i+1}^{j} q_k  \right\}=1-P(i, \infty) 
	\end{alignat*}

	We define $M$ as the Markov chain \textit{embedded} in success epochs of the Bernoulli trials. A Markov chain is said to be $\TP{r}$ if its transition matrix is $\TP{r}$ (see p. 42 in \cite{KarlinMarkov}). We now show that $M$ satisfies this property.

	\begin{thm}\label{prop:tp.stochastic.operator}
		The Markov chain $M$  is $\TP{}$.
	\end{thm}
	\begin{proof}
		Multiply each row $i\in S$ of  $P$ by $\prod_{k=1}^{i-1}q_{k}$ and divide  each column   $j \in S$ of the resulting matrix by $ p_j\prod_{k = 1}^{j-1} q_k $, to obtain the matrix $O \coloneqq D_1PD_2^{-1}$, where
		\begin{equation}\label{def.ABC}
			\begin{aligned}
				D_1 &= \diag(1, q_1, q_1q_2. \ldots, \prod_{k=1}^{n}q_{k}),\\
				D_2 &= \diag(1, p_1,  p_2q_1, \ldots,  p_{n}\prod_{k=1}^{n-1}q_{k}),\\[5pt]
				O&\coloneqq D_1PD_2^{-1} = \left(\begin{array}{cccccc}
					0 & 1 & 1 & \cdots& 1\\
					0 & 0 & 1 &   \cdots & 1\\
					\vdots & \ddots  &\ddots &\ddots & \vdots\\
					\vdots & \ddots  &\ddots &\ddots & 1\\		
					0 & 0 & \cdots  &0 &1 \\				
				\end{array}\right),
		\end{aligned}
	\end{equation}
	with $n \ge 2$ and $p_\infty = 1$ by convention. To complete the proof, it suffices to demonstrate that \( D_1^{-1} \), \( D_2 \), and \( O \) are \( \TP{} \), since \( P = D_1^{-1} O D_2 \) and the fact that \( \TP{} \) matrices are closed under matrix multiplication. The diagonal matrices \( D_1^{-1} \) and \( D_2 \) are trivially \( \TP{} \). For the matrix $O$, if we exclude the terminal state, the transpose of the remaining submatrix is the kernel representation of the sequence \( (0,1,1,\dots,1)\).  But this  is a constant sequence with a leading zero, hence the corresponding  kernel representation is  \( \TP{} \). Including the terminal state adds a row and column that preserve nonnegativity of all minors. Since all minors of all orders are nonnegative; hence   \( O \) is \( \TP{} \), completing the proof.

\end{proof}

\begin{lemma}\label{lemma.nth.transition.quasi}
	For each $n \in \mn$, the $n$-step transition matrix of $M$ preserves quasi-unimodality.
\end{lemma}
\begin{proof}
	Because both the class of totally positive matrices and the class of stochastic matrices are closed under multiplication,  \( P^n \) is a stochastic \( \TP{} \) matrix  for every \( n \in \mathbb{N} \). The claim follows directly from Theorem \ref{thm.unimodality}.
\end{proof}	

\section{Optimality of the myopic strategy}\label{sec4}
In \cite{FergusonOdds}, Ferguson introduced a general framework for best-choice problems using the  {martingale method} in the sense of \cite{Peskir}. He applied this framework to extend Bruss' Odds Theorem \cite{Odds} to an infinite number of trials. Building on Ferguson's  approach, \cite{TamakiOdds}, followed by  \cite{MatsuiAno}, established the optimality of the myopic strategy for the $m$th last-success problem and the $\ell$th-to-$m$th last-success problem, respectively. In this section, we rely on the Markovian structure of the Bernoulli stopping problem to derive fundamental results connecting  the quasi-unimodality of the reward sequence, the unimodality of the payoff sequence, and  the optimality of the myopic strategy. 

 Unless otherwise stated,   $f \coloneqq(f_k)_{ k \in S}$ and $g\coloneqq (g_k)_{ k \in S}$ denote the stopping reward and continuation payoff sequences of the Bernoulli stopping problem, where  $S = \mn \cup \{0, \infty\}$.
\subsection{Existence}

To establish the existence of an optimal stopping rule $\tau$ via the principle of optimality, we must verify that  $\mE{Y_{\tau}}$ exists. This  can be achieved by ensuring  that $\sup_k {Y_k}$ has a finite expectation. According to the general facts of optimal stopping theory, an optimal rule exists  under the following two conditions  (see Chapter 3 in \cite{FergusonBook} or p.2 in \cite{Peskir}): 
\begin{alignat}{2}
	\mEg { \sup_k \abs{Y_k}} < \infty, \tag{A1} \label{eqn.A1}\\ 
	\limsup_{k \to \infty} Y_k \le Y_\infty \text{ a.s.} \tag{A2}\label{eqn.A2},
\end{alignat}

\begin{thm}\label{def.existence}
	
	Consider a Bernoulli stopping problem with success profile $p$ and a bounded  stopping reward sequence $f$. Suppose  $f_\infty \ge  0$ is the reward for never stopping. Assumptions (\ref{eqn.A1}) and (\ref{eqn.A2}) hold provided one of the following conditions is satisfied 
	\begin{enumerate}[label=\rm(\roman*)]
		\item$ \sum_k \abs{f_k} < \infty$,
		\item $\sum_k p_k < \infty$, 
		\item $\sum_k {\abs{f_k}}p_k < \infty$ and $\limsup_{k \to \infty} \abs{{f_k}} \le f_\infty$.
	\end{enumerate}
\end{thm}
\begin{proof}
	From (\ref{def.payoof}), 
	$\sup_k {{Y_k}} = f_j^+ \coloneqq \max(f_j, 0)$ with probability $p_j$ for each $j \ge 1$.  Thus,
	\begin{alignat}{2}
		\mEg{\sup_k {{Y_k}}} \coloneqq  \sum_{j \ge 1}   {f_j}^+ {p_j}
		&\le  \sum_{j \ge 1}   \abs{f_j} {p_j},  \label{def.ineq.expectation.basic}\\
		&< \max( \sup_k \abs{f_k}, 1) \min\bigg( \sum_j \abs{f_j},  \sum_j p_j\bigg)		\label{def.ineq.expectation},
	\end{alignat}
	Condition (\ref{eqn.A1}) follows from \rm(i) or  \rm(ii) in view of (\ref{def.ineq.expectation}),  or from \rm(iii) via (\ref{def.ineq.expectation.basic}). Since  $Y_\infty = f_\infty$ a.s.,  condition (\ref{eqn.A2}) holds whenever $\limsup_{k \to \infty} f_k \1(X_k = 1) \le f_\infty$ a.s. This is guaranteed if either $\limsup_{k \to \infty} {f_k}  \le f_\infty$ or  $\lim\sup_{k \to \infty} X_k =0$ a.s.  The first instance is explicitly required in \rm(iii) and follows from $\abs{f_k} \to 0$ in \rm(i) 
	 since the series is absolutely convergent. The second instance, $\lim\sup_{k \to \infty} X_k =0$ a.s.,  is implied by  \rm(ii)  via the first Borel-Cantelli lemma. 
	 This concludes the proof.
\end{proof}

\noindent Having disposed of the existence step, we now examine the optimality of the myopic strategy, assuming throughout that $f$ is  nonnegative. Before proceeding, we recall the following fundamental result.
\begin{thm}[\cite{CRS, FergusonBook}]\label{thm.monotone.finite}
	In a finite-horizon monotone stopping rule problem, the myopic strategy is optimal.
\end{thm}

\subsection{The sufficient condition of Ferguson}

\begin{thm}[Ferguson's Theorem]\label{thm.ferguson}
	Consider a Bernoulli stopping problem with nonnegative stopping reward and  continuation payoff sequences $f$ and $g$, respectively. Assume the hypotheses of Theorem \ref{def.existence} hold. 	A sufficient  condition for the myopic strategy to be optimal is that the ratio $g/f$ (or equivalently, the  difference $g-f$) is nonincreasing.
\end{thm}	
\begin{proof}
	At stage $k \in \mn$, if $f_k = 0$, then stopping is not optimal, since we can do at least as well by continuing one more step. Consequently, we define both $g_k/f_k$	and $g_k-f_k$ to be $+\infty$ when $f_k=0$ even if $g_k=0$. The myopic policy calls for stopping at stage $k$ if  $g_k - f_k \le 0$, which is equivalent to  $g_k / f_k \le 1$. Define the events $B_k \coloneqq \{ X_k =1 \text { and } g_k - f_k\le 0 \}= \{ X_k =1  \text { and } g_k/ f_k \le 1 \}.$  Recall that the problem is  monotone \cite{CRS} if $B_1 \subset B_2 \subset B_3 \subset \cdots \text{ a.s.}$
	By assumption, $g_k / f_k \le g_{k+1} / f_{k+1}$ for all $k$, so $g_{k+1} / f_{k+1} \le 1$ implies  $g_{k} / f_{k} \le 1$. Likewise,  since $g_k - f_k \le g_{k+1} - f_{k+1}$ holds for all $k$, the condition $g_{k+1} \le f_{k+1} $ implies $g_{k} \le f_{k}$. 
	Thus, the sequence  $(B_k)$ is absorbing, and the myopic strategy is optimal.
%
\end{proof}
Theorem 1 in \cite{FergusonOdds} is a special case of the following corollary.
\begin{cor}\label{cor.ferguson}
	Suppose there exists a scalar $w \in [0,1)$ and a real sequence  $h\coloneqq(h_k)_{k \in \mn}$ such that $g = h + w f $. A sufficient  condition for the myopic strategy to be optimal is that the ratio $h/f$  is nonincreasing.
\end{cor}	
\begin{proof}
	If 	$w \ge 1$, stopping is never optimal since $f < g$, justifying the constraint on $w$. The result follows immediately from Theorem \ref{thm.ferguson}, since  $\Delta \displaystyle\frac{g}{ f}=\Delta \displaystyle\frac{h}{ f} $, where $\Delta$  is the difference operator.
\end{proof}

\subsection{Approximation of the infinite-horizon problem}
%
Suppose  there exists an optimal rule $\tau$  with optimal value  $V^*$.  Define its truncated form by  $\tau(n)=\min \{\tau, n\}$ for $n \in \mn$ with value $V^*_n$. One approach to tackle the infinite-horizon problem is to solve its truncated version  with $n$ trials then show that $\tau = \lim_{n \to \infty} \tau(n)$ and  $V^* = \lim_{n \to \infty} V^*_n$.   The topic is treated extensively in Chapter 5 of Ferguson's monograph \cite{FergusonBook}. This section provides details specific to the Bernoulli stopping case and discusses the conditions under which the approximation approach is valid.

We begin by verifying a uniform integrability condition, which will play a central role in the proofs.
  \begin{define}[Uniform integrability  \cite{CRS, FergusonBook}]
	Let  $T \coloneqq (T_k)_{k\ge 0}$ be a sequence of real random variables. If either of the following  holds:
	\begin{enumerate}
		\item $\lim_{a \to \infty} \sup \mE{ \1\{ \abs{T_n }> a\} } = 0,$
		\item $\sup \mE{ \abs{T_n}} < \infty$ and  $\mE{ \1\{A_n\}\abs{T_n}} \to 0$ as $n \to 0$, for  any $A_n$ such that $\prob(A_n) \to 0$ as $n \to \infty$,
	\end{enumerate}
	then $T$ is \textit{uniformly integrable}.
\end{define}
Define the sequence of random variables $T \coloneqq (T_k)$, where
\begin{equation}\label{def.T}
	T_k \coloneqq \sup_{j \ge k} \{ Y_j - Y_k\},
\end{equation}
and, as before, $Y_k=X_kf_k$ is the random payoff for stopping at trial $k$.

\begin{lemma}\label{lem.uniformT}
	If  $\sum_k p_k< \infty$ or $\sum_k \abs{f_k}< \infty$ then $T$ is uniformly integrable.
\end{lemma}
\begin{proof}
	For all $n \in \mn$,  $\mE{ \1(\abs{T_n} > a)} \le \mE{ \1(\sup_n \abs{Y_n} > a)} \le  {\mE{ {\sup_n \abs{Y_n}}}}/{a}$ by Markov inequality. By (\ref{def.ineq.expectation}),  $\mE{\sup_n \abs{Y_n}} < C\min( \sum_k p_k, \sum_k \abs{f_k})$ for some constant $C$. Hence, $\mE{ \1(\abs{T_n} > a)} \to 0$  as $a \to \infty$, proving the assertion.
\end{proof}
\begin{lemma}[Lemma 1 in Chapter 5, \cite{FergusonBook}]\label{lem.vanish}
	Suppose $T$ is a uniformly integrable sequence. If $(A_n)$ is any sequence of  events such that $\prob(A_n )\to 0$ as $n \to \infty$, then as $n \to \infty$, $\mE{\1\{ A_n\} \abs{T_n} } \to 0$.
\end{lemma}

\begin{thm}\label{thm.convergence}
	Suppose $f$ is nonnegative and $f_\infty=0$. 
If  $\sum_k p_k< \infty$ or $\sum_k {f_k}< \infty$, then	as $n \to \infty$, $\lim_{n \to \infty} V^*_n \to V^*$.
\end{thm}
\begin{proof}
	By hypothesis $f \ge 0$ and $f_\infty = 0$, therefore the random reward sequence $Y$ is nonnegative and 
	$\mE{Y_\infty} = 0$, respectively. By Theorem \ref{def.existence},   (\ref{eqn.A1}) and (\ref{eqn.A2}) are satisfied and hence an optimal rule exists.  We denote this rule $\tau$, with $\tau(n)=\min \{\tau, n\}$ for $n \in \mn$ being its  truncated form. Then
	$$
	\begin{aligned}
		0 \leq V^*-V^*_n & \leq \mE{Y_\tau}-\mE{ Y_{\tau(n)}} \\
		& =\mE{ \1\{ n<\tau<\infty\} \left(Y_\tau-Y_n\right)} +\mE{ \1\{ \tau=\infty\} \left(Y_\infty-Y_n\right) } \\
		& \le \mE{ \1\{ n<\tau<\infty\} \abs{T_n}} +\mE{Y_\infty} = \mE{ \1\{ n<\tau<\infty\} \abs{T_n}} 
	\end{aligned}
	$$
Since  $T \coloneqq (T_k)$ defined in (\ref{def.T})  is uniformly integrable  by Lemma \ref{lem.uniformT}, then
as $n \to \infty$, $\prob(\1\{ n<\tau<\infty\}) = 0$. Invoking Lemma \ref{lem.vanish} concludes the proof.
\end{proof}
This result demonstrates only convergence of the optimal value; it does not ensure that the myopic strategy remains optimal in the infinite-horizon setting. The next theorem and  its corollary furnish this guarantee.

\begin{thm}[Theorem 2, Chapter 5 in \cite{FergusonBook}]\label{thm.convergene.myopic}
	Suppose (\ref{eqn.A1}) and (\ref{eqn.A2}) are satisfied and suppose the problem is monotone. If  $V^*_n \to V^*$ as $n \to \infty$, then the myopic strategy is optimal.
\end{thm}

\begin{cor}\label{cor.monotone}
	Suppose $f$ is nonnegative and $f_\infty=0$. 
		Assume  the hypothesis of Theorem \ref{thm.convergence} is satisfied. 
	If the stopping problem is monotone then the myopic strategy is optimal. 
\end{cor}
\begin{proof}
The  assertion  follows from Theorem \ref{thm.convergene.myopic} 	since  $\lim_{n \to \infty} V^*_n \to V^*$ holds by Theorem \ref{thm.convergence}.
\end{proof}

This result is unsurprising, but monotonicity remains essential for proving optimality of the myopic rule and for passing to the infinite-horizon limit.  In the next section, we show that this condition is ensured by quasi-unimodal stopping rewards or unimodal continuation payoffs.
\subsection{Unimodality of the Continuation Payoff Sequence}
Ferguson's theorem provides a sufficient condition for optimality without imposing restrictions on the monotonicity patterns of the reward or payoff sequences. However,  it does not exploit the  structural relationship between the properties of $f$  and  $g$ inherent in Bernoulli stopping problems.  In the classical last-success problem, \cite{Odds} demonstrated  that monotonicity  follows implicitly from the unimodality of  $g$. Moreover,  \cite{FergusonDependent} argued that the unimodality of $g$, within the framework of best-or-worst choice problems and the bivariate secretary problem, ensures the optimality of the myopic policy. Despite these insights, the direct connection between  $f$ and $g$ was not utilised.  In what follows, we show how unimodal continuation payoffs guarantee the optimality of the myopic strategy.
%
%
%
%

\begin{thm} \label{thm.unimodal.myopic}
	Consider the Bernoulli stopping problem in which success epochs are governed by a Markov chain with transition matrix (\ref{def.embedded.chain.poissonbinom}). If the continuation payoff is unimodal, then the stopping problem is monotone.
\end{thm}
\begin{proof}
	For any set $E\subset \{1,2,\ldots,\infty\}$, there exists a strategy that stops immediately when the Markov chain $M$ enters $E$. Conditional that $M$ ever entering $E$, the first visited  site $\min\{k\in E: X_k=1\}$ is independent of the path of $M$ prior to hitting $E$. By definition, $g_{k-1}$ is the expected payoff of the strategy associated with the stopping set $E_k:=\{k, k+1,\ldots,\infty\}$, given by
	\begin{equation}\label{eqn.expected.ola.payoff}
		g_{k-1}=\sum_{j\ge k} P(k-1,j)f_j.
	\end{equation}
	By monotonicity, once the chain enters \(E_k \), it cannot leave. 	The first site visited in $E_k$ is either $k$ or an element of $E_{k+1}$, leading to the recursion
	\begin{equation}\label{recursion.g.f}
		g_{k-1}=p_k f_k + (1-p_k)g_{k},
	\end{equation}
	which implies
	\begin{equation}\label{eqn.inequality.ola.payoff}
		g_{k-1}\geq g_{k}\Longleftrightarrow  f_k\geq g_{k}.
	\end{equation}
	Because $g$ is unimodal, this equivalence ensures the existence of a threshold $k_*$ such that $f_k \ge g_k$ for all $k \ge k_*$. By the proof of Theorem \ref{thm.ferguson},   the sets $(A_k)$, where $A_k \coloneqq \{ X_k =1 \text { and } g_k - f_k \le 0 \},$ are absorbing. Consequently, the problem is monotone, and the myopic strategy is optimal.
\end{proof}

\begin{cor}\label{cor.threshold}
	The myopic strategy  is a threshold strategy, where the optimal value equals  the mode of the continuation payoff and the optimal threshold corresponds to the index of that mode.
\end{cor}
\begin{proof}
	Let $E_k:=\{k, k+1,\ldots,\infty\}$ and suppose the mode interval of $g$ is $[k_-, k_+]$. Using (\ref{eqn.inequality.ola.payoff}), the optimal stopping set is given by   
	\begin{equation}\label{eq.stoppingset}
B \coloneqq \{j \in \mn: ~X_j =1 \text { and } f_j-g_j \ge 0 \} = \{ j \in E_{k_-}: X_j = 1 \}
	\end{equation}
	By monotonicity,  the myopic rule  coincides with the first visit to $B$, namely
	$\min \{j \ge k_-: X_j=1\}$. But this defines a threshold strategy with threshold \( k_- \) and expected payoff \( g_{k_-} \), completing the proof.
\end{proof}

\subsection{Quasi-Unimodality of the Stopping Payoff Sequence}

In the context of the relatively best-choice problem with a fixed number of trials, Lemma 2 in \cite{FergusonHardwick} established via backward induction that the unimodality of the stopping payoff \( f \) secures the optimality of the myopic strategy. Since this proof is independent of the success profile, it remains valid for Bernoulli stopping problems. The authors also demonstrated that the unimodality of \( f \) ensures the optimality of the myopic strategy in the discounted duration model (see p. 44 in \cite{FergusonHardwick}). Because the monotone case of optimal stopping driven by the (quasi-)unimodality  of \( f \) has not been extensively explored, we extend this result to Bernoulli stopping problems.

The following elementary lemma will be useful in proving our main result,

\begin{lemma}\label{lem.sup}
Let $\alpha \coloneqq(\alpha_k)_{k \in \mn}$ be a bounded sequence of nonnegative numbers. If  $\alpha_k \to 0$ as $k \to \infty$, then  $\alpha$ attains its maximum.
\end{lemma}
\begin{proof}
	Let $M \coloneqq \sup \alpha$. Clearly, $M \in [0, \infty)$ since $a$ is nonnegative and bounded by assumption. Because $\alpha_k \to 0$,  there exists $N \in \mathbb{N}$  such that $\alpha_k < M/2$ for all $k > N$. This implies that $\sup \alpha =\sup\bigl\{\alpha_1, \alpha_2, \ldots, \alpha_N\bigr\}$. Since this set is finite, its supremum coincides with the maximum. Thus, $\alpha$ attains its maximum.
\end{proof}
\begin{thm}\label{thm.unimodality_stopping_reward}
	Consider the Bernoulli stopping problem with nonnegative and bounded stopping reward $f$, where $f_\infty=0$. Assume $\sum_k p_k< \infty$, where $p_k \in (0,1 )$. 	If $f$ is quasi-unimodal, then the stopping problem is monotone.
	
\end{thm}

\begin{proof}
	By Lemma \ref{lemma.nth.transition.quasi}, $g=Pf$ is quasi-unimodal. To prove  the unimodality of $g$, is it sufficient to show that it attains its supremum.  Let $M \coloneqq \sup g$. Clearly,  $0 \le M = \sup (Pf) \le \sup f < \infty$. If $M=0$, then $g=0$ and the proof is complete. Suppose $M > 0$. For $i \in S$, we have
	\begin{alignat*}{2}
		\lim_{i \to \infty} g_i = \lim_{i \to \infty}\sum_{j=i+1}^\infty P(i,j) f_j &\leq\sup_{j \geq i+1} f_j  \left( \lim_{i \to \infty}\sum_{j=i+1}^\infty P(i,j) \right) \\
		&= \sup_{j \geq i+1} f_j \left( \lim_{i \to \infty}\sum_{j=i+1}^\infty p_j \prod_{k=i+1}^{j-1} (1 - p_k)\right) \to 0
	\end{alignat*}
Since $g$ is nonnegative, invoking Lemma \ref{lem.sup} completes  the proof.
\end{proof}

\begin{cor}\label{cor.quasi_unimodality_nstopping}
Assume $n \in \mn$.	If $f$ is quasi-unimodal, then the $n$-step continuation payoff	$P^nf$  is unimodal.
\end{cor}
\begin{proof}
	Obvious by induction on $n$.
\end{proof}

\begin{lemma}\label{lem.quasi_unimodality_sum_nstopping}
Let $m_1, m_2, \ldots, m_N$ be $N \ge 1$ arbitrary positive integers. 	If $f$ is nonincreasing on $S \setminus \{\infty\}$, then the  payoff	$(P^{m_1} + P^{m_2} +  \cdots + P^{m_N})f$ is unimodal.
\end{lemma}
\begin{proof}
 Fix $N \ge 1$ positive integers $m_1, m_2, \ldots, m_N$ and  let  $g^{(i)} = P^i f $ for $i \in \{ m_1, m_2, \ldots, m_N\}$. By the previous corollary, each $g^{(i)}$ is unimodal. Hence, $h \coloneqq (P^{m_1} + + P^{m_2} +  \cdots + P^{m_N})f = g^{(m_1)} +  \cdots + g^{(m_N)}$ attains its maximum. It remains to show that $h$ possesses at most two sign changes. Since $f$ is nonincreasing on $S \setminus \{\infty\}$, the VD property of $P$ implies that each $g^{(i)}$  is nonincreasing. It follows that  $h$ is nonincreasing on $S \setminus \{\infty\}$. Because $h$ is nonnegative, it remains quasi-unimodal on $S$, which is the desired result.
\end{proof}

\begin{rem}
	$ $
	\begin{enumerate}[label=\rm(\roman*)]
		\item 	Theorem \ref {thm.convergence} is a slight modification of Theorem 3 in chapter 5 of \cite{FergusonBook}, where the assumption $\limsup_n Y_n = Y_\infty$ a.s. is replaced by requiring  $f$  to be nonnegative and $f_\infty=0$.
		
		\item In the setting with infinitely many successes, an optimal strategy may still exist as seen under the assumptions of Theorem \ref{thm.unimodal.myopic}.  In this case, $\sum_k p_k = \infty$ is equivalent to $\tau<\infty$ a.s. \rm(i) or \rm(iii) of Theorem  \ref{def.existence} may or may not hold. 

		\item A continuous-time analogue of the problem is obtained by allowing  the player to stop at an epoch of a Poisson process with locally integrable rate function $p(t)$.
		
		\item In the finite-$n$ Bernoulli stopping problem of \cite{Ribas2019}, $f_k=w_kP(k,\infty)$ and \( w_k \) is the reward for stopping on a success at stage \( k \in \{ 1, 2, \ldots, n\} \). Proposition 6 in that paper shows that  $f-g$ changing sign at most once is necessary and sufficient for monotonicity, which is interpreted  as the stopping set consisting of a single interval. This does not extend to the infinite-horizon case, since  \( f < g \) implies that stopping is never optimal.
	\end{enumerate}
\end{rem}

\section{Applications}\label{sec5}
Consider the last-success problem, where the objective is to select any of the \( \{\ell\)th, \( (\ell+1)\)th, \(\dots\), \( m\)th\( \} \) last successes in \( n < \infty \) independent Bernoulli trials \( X_1, X_2, \dots, X_n \), with \( \ell \leq m \) \cite{MatsuiAno}. We extend this problem to the case of an infinite number of trials. The optimisation problem seeks the rule satisfying
\begin{alignat}{2}\label{def.optimal.rule}
	\tau^* &\coloneqq \argmax_{\tau \in {\cal T}}~ \prob\left[\ell \le \sum_{j \ge \tau} X_j \le m\right].
\end{alignat}
Without loss of generality, assume that $p_k \in [0,1)$  for all $k\ge1$, where $p_k$ is the success probability of the $k$th trial. If $\sum_{k} p_k = \infty$, then by the Borel-Cantelli lemma, $\sum_k X_k = \infty$ a.s. and no optimal rule exists. For the remainder of this section, we assume $\sum_k p_k < \infty$.

For fixed $n < \infty$, the number of successes among trials $k, k+1,\cdots, n$ is a random variable 
following a Poisson-binomial distribution with probability-generating function
$$z \mapsto \prod_{j=k}^n (1-p_j+zp_j).$$
Define the odds ratio $r_{k}\coloneqq {p_k}/{(1-p_k)}$.   A straightforward computation yields the probability of zero successes and the probability of $m$ successes from stage $k$ onward:
\begin{equation}
	\begin{aligned}
		s_0(k, n)&\coloneqq\prod_{j = k}^n (1-p_j),\nonumberj
		s_m(k, n)&\coloneqq s_0(k, n)e_m(k, n), \nonumberj
	\end{aligned}
\end{equation}
where $e_m(k, n)$ is an elementary symmetric polynomial
\begin{equation*}
	e_{m}(k, n)\coloneqq\sum_{k \le i_1 < \ldots <i_m \le n} r_{i_1}\cdots r_{i_m}, ~~~e_{0}(k, n)=1.
\end{equation*}
For a given $\ell \in [1,m]$, independence of the trials implies that the probability of observing $j \in [\ell, m]$  successes  from trial $k$ onward is
\begin{alignat}{2}\label{uml}
	u_m(k, n; \ell) \coloneqq \sum_{j=\ell}^m s_j(k, n).
\end{alignat} 
From (\ref{def.optimal.rule}), it is clear that maximising the probability of selecting a success amongst the  $\ell\text{th}, (\ell+1)\text{th}, \ldots, m\text{th}$ last successes is equivalent to maximising $u_m(\cdot, n,\, \ell)$. For an infinite number of trials,  we drop the second argument of $s_m$, $u_m$, $e_m$ and define $s_0(\infty)=0$.  Moreover, if $\sum_k p_k < \infty$ then $e_m(k)$ converges for all $k$ since  $e_m(1) \le  e^m_1(1) < \infty$.  

The next proposition is key to the solution of various last-success problems in  infinite-horizon.

\begin{assertion}\label{prop.unimodal.sm}
	Assume $\sum_k p_k < \infty$. 	For any $m < n  \in \mn$, the sequences $(s_m(1), s_m(2), \ldots)$ and $(u_m(1; \ell), u_m(2; \ell), \ldots)$ are unimodal. 
\end{assertion}
\begin{proof}
Fix $m < n$	and introduce the sequences $a=(a_1, a_2, \ldots)$, $b\coloneqq (b_1, b_2, \ldots)$ and , $c\coloneqq (c_1, c_2, \ldots)$, where  
	$ a_k \coloneqq 	s_{0}(k) $,  $ 	b_k = s_{m}(k)$ and $c_k = u_{m}(k; \ell)$ . These  sequences satisfy $b=P^ma$,  
	and  $c=P^\ell a + \cdots +P^m a $, where $P$ is the transition matrix  in (\ref{def.embedded.chain.poissonbinom}).  Since $a$   is nondecreasing, it is quasi-unimodal; hence  $b$ is  unimodal by application of Corollary \ref{cor.quasi_unimodality_nstopping}. Similarly, $c$  is  unimodal by application of Lemma \ref{lem.quasi_unimodality_sum_nstopping}.
\end{proof}
\begin{lemma}\label{lem.new}
Let $I \coloneqq \{m_1, m_2, \ldots, m_N\}$ be a arbitrary collection of positive integers, where $N \in \mn$.  Define 
	\begin{alignat}{2}\label{v_I}
		v(k; I) \coloneqq \sum_{j \in I} s_j(k),\quad k \in \mn.
	\end{alignat} 
If  $\sum_k p_k < \infty$, then $v$ is unimodal.
\end{lemma}
\begin{proof}

The Lemma is proved in much the same way as  Proposition \ref{prop.unimodal.sm}. Given an arbitrary set $I$ with $N \ge 1$ positive integers $m_1, \ldots m_N$, introduce the sequences $a=(a_1, a_2, \ldots)$ and $d = (d_1, d_2, \ldots)$, where $ a_k \coloneqq 	s_{0}(k) $, and $d_k = v(k; I)$. From the definition of $v(\cdot; I)$, we have $d=P^{m_1} a + \cdots +P^{m_N} a $.  Since $a$   is  quasi-unimodal, the assertion is proved  by application of Lemma \ref {lem.quasi_unimodality_sum_nstopping}.
\end{proof}
\subsection{Selecting the $m$th Last Success}
The optimisation task in this last-success problem is  defined in (\ref{def.optimal.rule}) with $\ell=m$. 
We generalise Theorem 2 in \cite{BrussPaindaveine}  to the infinite-horizon setting as follows.
\begin{thm}\label{thm:unimodality.lastm.infinite}
	Consider the {$m$th  last-success} problem in an infinite sequence of independent Bernoulli trials. Suppose  $\sum_k p_k < \infty$ and  that zero payoff is awarded for never stopping. Then an optimal rule exists,  and the optimal strategy is to stop at the first success, if any, occurring on or after index
	\begin{alignat*}{2} 
		k^* \coloneqq \min\{j \in \mn: e_m(j) \le  e_{m-1}(j)\}.
	\end{alignat*}
	The optimal value of the problem is  $s_m(k^*)$.		
\end{thm}
\begin{proof}
Define the sequences  $f=(f_k)_{k \in S}$ and $g \coloneqq (g_k)_{k \in S}$, where  
$$
f_k = \begin{dcases}
	s_{m-1}(k+1) & k \ge 1 \\
	0 & k=0 \text{ or } k = \infty
\end{dcases}, \text{ and }~
g_k = \begin{dcases}
	s_{m}(k+1) & k \ge 1 \\
	0 & k=0 \text{ or } k = \infty
\end{dcases}.
$$
These sequences are the payoffs for stopping at stage $k$ with  respectively $m-1$ successes and $m$ successes obtained from stage $k+1$ onwards.  They satisfy the relation $g=Pf$ where $P$ is the transition matrix defined in (\ref{def.embedded.chain.poissonbinom}).  By Proposition \ref{prop.unimodal.sm}, both $f$ and $g$ are unimodal and hence  the problem is monotone according to Theorem \ref{thm.unimodal.myopic}. By Corollary \ref{cor.threshold},  the optimal strategy is a threshold rule, with the threshold  occurring at the first index where $g_{k-1} - g_k$ changes sign from $-$ to $+$. Using (\ref{eqn.inequality.ola.payoff}), this occurs at $k^*= \min \{ j \in \mn: s_{m-1}(j) \ge s_{m}(j)\} = \min \{ j \in \mn: e_{m-1}(j) \ge e_{m}(j)\}$ and the optimal value is $g_{k^*}$.
	\end{proof}

\subsection{Selecting any success between $\ell$th last and $m$th last}
For the generalisation of the previous problem, the \( \ell \)th-to-\( m \)th last-success problem with a finite number of trials, the threshold rule is expressed in terms of a ratio of elementary symmetric polynomials (multiplicative odds). Its optimality is established by Theorem 2 in \cite{MatsuiAno}. We extend their result to the infinite case, following the same approach as before.

\begin{thm}\label{thm:existence.any_l_to_m}
	Consider the {$\ell$th-to-$m$th  last-success} problem in an infinite sequence of independent Bernoulli trials. Suppose $\sum_k p_k < \infty$ and that zero payoff is awarded for never stopping. Then, an optimal rule exists, and the  strategy is to stop at the first success, if any, occurring on or after index 
	\begin{alignat}{2} \label{threshold.mth}
		l^* \coloneqq \min\{j \in \mn: e_{m-1}(j) \le  e_{\ell}(j)\}.
	\end{alignat}
	The optimal value of the problem is \( u_m(l^*;\, \ell) \).
\end{thm}

\begin{proof}
	The argument parallels that of Theorem \ref{thm:unimodality.lastm.infinite}. 	Define the sequences  $f=(f_k)_{ k \in S}$ and $g\coloneqq Pf=(g_k)_{ k \in S}$, where 
	$$
	f_k = \begin{dcases}
		u_{m-1}(k+1; \ell-1) & k \ge 1 \\
		0 & k=0 \text{ or } k = \infty
	\end{dcases}, \text{ and }~
	g_k = \begin{dcases}
		u_{m}(k+1; \ell) & k \ge 1 \\
		0 & k=0 \text{ or } k = \infty
	\end{dcases}.
	$$
	By Proposition \ref{prop.unimodal.sm}, the sequence $g$ is unimodal and hence  the problem is monotone according to Theorem \ref{thm.unimodal.myopic}. By Corollary \ref{cor.threshold},  the optimal strategy is a threshold rule, with the threshold  occurring at the first index where  $g_{k-1} - g_k$ changes sign from $-$ to $+$. Using (\ref{eqn.inequality.ola.payoff}),  this occurs at  $l^*= \min \{ j \in \mn: u_{m-1}(j; \ell-1) \ge u_{m}(j; \ell)\} = \min \{ j \in \mn: s_{m-1}(j) \ge s_{\ell}(j)\} = \min \{ j \in \mn: e_{m-1}(j) \ge e_{\ell}(j)\}$ and the optimal value is $g_{l^*}$.
\end{proof}

\begin{rem}
From Lemma \ref{lem.new}, it is immediate that the Bernoulli stopping problem in which stopping is allowed on an arbitrary collection of success indices is also monotone and the myopic strategy is optimal.
\end{rem}

\section*{Acknowledgement}
I wish to express my gratitude to Sasha Gnedin, who suggested the Bernoulli stopping problem and provided several valuable suggestions. 
\printbibliography

\end{document}